\documentclass[12pt]{article}

\usepackage{amsfonts}
\usepackage{amssymb}
\usepackage{amsmath}
\usepackage{amsthm}
\usepackage{graphicx}
\usepackage{eucal}
\usepackage{ wasysym }
\usepackage{tikz}
\usepackage{tkz-berge} 
\usetikzlibrary{arrows,decorations.markings}
\usepackage{hyperref}
\hypersetup{colorlinks=true}
\usepackage{verbatim}
\usepackage{todonotes}
\usepackage{adjustbox}

\theoremstyle{definition}
\newtheorem{lem}{Lemma}[section]
\newtheorem{prop}[lem]{Proposition}
\newtheorem{thm}[lem]{Theorem}

\newtheorem*{Def}{Definition}
\newtheorem{alg}{Algorithm}

\begin{document}

\title{The Lights Out Game on Directed Graphs}

\author{T. Elise Dettling\thanks{Department of Mathematics, Grand Valley State University, 
Allendale, Michigan
49401-6495, dettlint@mail.gvsu.edu,}
\and Darren B. Parker\thanks{Department of Mathematics, Grand Valley State University, 
Allendale, Michigan
49401-6495, parkerda@gvsu.edu,
\url{http://faculty.gvsu.edu/parkerda}
}}

\date{
\small MR Subject Classifications: 05C20, 05C40, 05C50, 05C57, 05C78}

\maketitle

\begin{abstract}
We study a version of the lights out game played on directed graphs.  For a digraph $D$, we begin with a labeling of $V(D)$ with elements of $\mathbb{Z}_k$ for $k \ge 2$.  When a vertex $v$ is toggled, the labels of $v$ and any vertex that $v$ dominates are increased by 1 mod $k$.  The game is won when each vertex has label 0.  We say that $D$ is $k$-Always Winnable (also written $k$-AW) if the game can be won for every initial labeling with elements of $\mathbb{Z}_k$.  We prove that all acyclic digraphs are $k$-AW for all $k$, and we reduce the problem of determining whether a graph is $k$-AW to the case of strongly connected digraphs.  We then determine winnability for tournaments with a minimum feedback arc set that arc-induces a directed path or directed star digraph.
\end{abstract}

\section{Introduction}

The lights out game was originally an electronic game created by Tiger Electronics in 1995.  The idea behind the game has since been extended to several light-switching games on graphs.  Some of these extensions are direct generalizations of the original game, like the $\sigma^+$-game in \cite{Sutner:90} and the neighborhood lights out game developed independently in \cite{paper11} and \cite{Arangala:12}.  This was generalized further to a matrix-generated version in \cite{paper15}.  Other versions are explored in \cite{Pelletier:87}, \cite{Craft/Miller/Pritikin:09}, and \cite{paper14}.

In each version of the game, we begin with some labeling of the vertices, usually by elements of $\mathbb{Z}_k$ for some $k \ge 2$.  We play a given game by toggling the vertices, which changes the labels of some vertices according to whether or not they are adjacent to the toggled vertex.  The game is won when we achieve some desired labeling, usually where each vertex has label 0.

In this paper, we study a directed graph version of the neighborhood lights out game.  The game begins with a digraph $D$ and an initial \emph{labeling} of the vertices with elements of $\mathbb{Z}_k$, which we primarily express as a function $\lambda: V(D) \rightarrow \mathbb{Z}_k$.  The game is played by toggling vertices.  Each time a vertex $v$ is toggled, the label of $v$ and each vertex it dominates (i.e. each $w$ such that $vw \in A(D)$) is increased by 1 modulo $k$.  The object of the game is to toggle the vertices so that all vertices have label 0.  We call this game the $k$-lights out game (note that in \cite{paper15}, we call this the $(N,k)$-lights out game, where $N$ is the neighborhood matrix of the graph; we drop the $N$ in our notation here, since we only play one version of the game in this paper).  If it is possible to win the game when we begin with the labeling $\lambda$, we call $\lambda$ a $k$-winnable labeling.  We say that $D$ is $k$-Always Winnable (or $k$-AW) if every labeling of $V(D)$ is $k$-winnable.

The basic question we explore is for which digraphs $D$ and natural numbers $k$ is $D$ $k$-AW.  In Section~\ref{acyclic}, we solve this problem for acyclic digraphs and reduce the problem of determining $k$-winnability on digraphs to determining $k$-winnability on strongly connected digraphs.  We then turn our attention to strongly connected tournaments.  After introducing some definitions and elementary facts about feedback arc sets in Section~\ref{feedback}, we determine in Section~\ref{strong} precisely when a strongly connected tournament is $k$-AW in the cases where a minimum feedback arc set arc-induces either a directed path or a directed star.

For the most part, we use notation and terminology as in \cite{Bang-Jensen/Gutin}.  A digraph $D$ consists of a set of \emph{vertices} $V(D)$ along with a set $A(D)$ of ordered pairs of vertices called \emph{arcs}.  For $v,w \in V(D)$, we can express the arc $(v,w)$ as $vw$ or $v \rightarrow w$.  In this arc, we call $v$ the \emph{tail} and $w$ the \emph{head} of the arc and say that $v$ dominates $w$. If $B \subseteq A(D)$, we define the subgraph \emph{arc-induced} by $B$ (denoted $D_B$) to have arc set $B$ and vertex set consisting of all vertices incident with an arc in $B$.

\section{Linear Algebra}

There are two ways that linear algebra inserts itself into the study of the neighborhood lights out game.  One way is analogous to the neighborhood game on graphs as studied in \cite{Anderson/Feil:98}, \cite{Arangala/MacDonald/Wilson:14}, \cite{paper11}, \cite{paper15}, and \cite{Hope:10}.  For a digraph $D$, suppose we order the vertices $v_1, v_2, \ldots , v_n$.  We can express an initial labeling $\lambda$ of $V(D)$ as a column vector $\textbf{b}$, where $\textbf{b}[j] = \lambda(v_j)$.  If each vertex $v_i$ is toggled $x_i$ times, we can also express this as a column vector $\textbf{x}$, where $\textbf{x}[i] = x_i$.  Since the label of each $v_j$ is increased by $x_i$ mod $k$ for each $v_i$ that dominates $v_j$ and is unchanged by each $v_i$ that does not dominate $v_j$, the resulting label for each $v_j$ is $\lambda(v_j)+\sum_{i=1}^n a_{ij}x_i$, where $a_{ij}=1$ if either $i=j$ or $v_iv_j \in A(D)$, and $a_{ij}=0$ otherwise.  We then determine how to win the game by setting each label equal to 0, which results in a system of linear equations.  The toggling $\textbf{x}$ wins the game precisely when $\sum_{i=1}^n a_{ij}x_i = -\lambda(v_j)$.  If we let $N=[a_{ij}]$, winning the $k$-lights out game is equivalent to finding a solution to the matrix equation $N\textbf{x} = -\textbf{b}$.  Note that if $A$ is the adjacency matrix of $D$, then $N=A+I_n$, which we call the \emph{neighborhood matrix} of $D$.  We get the following.

\begin{prop}
Let $D$ be a digraph, let $N$ be the neighborhood matrix of $D$, and let $\textbf{b}$ be an initial labeling of $V(D)$.
\begin{enumerate}
\item $\textbf{b}$ is $k$-winnable if and only if there is a solution to the matrix equation $N\textbf{x} = -\textbf{b}$.
\item $D$ is $k$-AW if and only if $N$ is invertible over $\mathbb{Z}_k$.
\end{enumerate}
\end{prop}
\noindent The second way that we encounter linear algebra is through the following process.
\begin{itemize}
\item We prove that given any initial labeling, we can toggle the vertices in such a way that we get a labeling that has a more desirable form.
\item Starting with a labeling in the desirable form, we determine vertices where the number of toggles is unknown, and assign each of the number of toggles of these vertices a variable.
\item We use the labeling that results from the toggling above to generate a system of linear equations.
\item We use linear algebra to determine when a solution exists.
\end{itemize}
This strategy often leads to a system of linear equations that has fewer variables or is otherwise simpler to solve than the system generated by the neighborhood matrix.

In both instances, we need to determine whether a system of equations has a solution regardless of the initial labeling.  That requires knowing when the matrix associated with that system of equations is invertible.  We use the following facts.

\begin{prop} \label{Brownprop} \cite{Brown:93}
Let $A$ be an $n \times n$ matrix over a commutative ring $R$.
\begin{enumerate}
\item $A$ is invertible if and only if $\det(A)$ is a unit in $R$.
\item If we take a multiple of one row of $A$ and add it to another row, then the determinant of the resulting matrix is $\det(A)$.
\item If we switch two rows of $A$, the resulting matrix has determinant $-\det(A)$.
\item If we multiply one row of $A$ by $r \in R$, the resulting matrix has determinant $r\det(A)$.
\end{enumerate}
\end{prop}

\section{Acyclic Digraphs and Strongly Connected Components} \label{acyclic}

Our techniques for determining whether the lights out game can be won depend heavily on a strategic ordering of the vertices.  All of our toggling strategies begin with some variation of the following algorithm.

\begin{alg}
\label{duh}
Given a digraph $D$,
\begin{itemize}
\item Order the vertices of $D$ in a particular way.  Say the ordering is $v_1,v_2, \ldots , v_n$.
\item Toggle $v_1$ until it has label 0.
\item Recursively, once $v_m$ has been toggled, toggle $v_{m+1}$ until it has label 0.
\item Continue toggling in this way until $v_n$ has been toggled.
\end{itemize}
\end{alg}
Most often, this algorithm by itself will not win the game, but it often gets us closer to determining whether or not the game can be won.

Recall that a \emph{walk} in $D$ is a sequence of vertices $v_1v_2 \cdots v_n$, where each $v_iv_{i+1}$ is an arc.  A \emph{cycle} is a closed walk (i.e. where $v_1=v_n$) where there are no repeated vertices except the first and last vertex.
The strategy suggested by Algorithm~\ref{duh} is quite effective on \emph{acyclic digraphs}, which are digraphs that contain no cycles.  It is straightforward to prove (see \cite[Prop. 2.1.3 (p. 33)]{Bang-Jensen/Gutin}) that every acyclic digraph has an ordering $v_1,v_2, \ldots , v_n$ of the vertices (called an \emph{acyclic ordering}) such that whenever $v_iv_j \in A(D)$, then $i<j$.  This helps us prove the following.

\begin{thm}
Let $D$ be an acyclic digraph.  Then $D$ is $k$-AW for all $k \ge 2$.
\end{thm}

\begin{proof}
Let $\lambda: V(D) \rightarrow \mathbb{Z}_k$ be any labeling of $V(D)$.  As noted above, there is an acyclic ordering $v_1,v_2, \ldots , v_n$ of $V(D)$. We apply Algorithm~\ref{duh} to this ordering.  No vertex can dominate any vertex that comes before it in the ordering, so every time we toggle a vertex, it does not affect the labeling of the vertices before it.  Thus, an easy induction proves that after toggling each $v_m$, this leaves $v_1,v_2, \ldots , v_m$ all having label 0.  Applying this result to $m=n$ gives us the zero labeling for $V(D)$, and thus $D$ is $k$-AW.
\end{proof}

We now look at connectivity in digraphs.  We say that a digraph $D$ is \emph{strongly connected} if for every $v,w \in A(D)$, there is both a walk from $v$ to $w$ and a walk from $w$ to $v$.  On any digraph $D$, a \emph{strong component} is a maximal subdigraph that is strongly connected.

\begin{prop} \label{components}
 \cite[p. 17]{Bang-Jensen/Gutin}
Let $D$ be a digraph, and let $D_1, D_2, \ldots , D_t$ be the strongly connected components of $D$.
\begin{enumerate}
\item \label{partition} The sets $V(D_1), V(D_2), \ldots , V(D_t)$ form a partition of $V(D)$.
\item \label{acyclicorder} We can order the $D_i$ in such a way that if $v \in V(D_i)$ and $w \in V(D_j)$ with $vw \in A(D)$ and $i \ne j$, then $i<j$.  This is called an \emph{acyclic ordering} of the strong components.
\end{enumerate}
\end{prop}

This helps us reduce the problem of determining winnability on digraphs to determining winnability on strongly connected digraphs.

\begin{thm} \label{strongreduce}
Let $D$ be a digraph, and let $D_1, D_2, \ldots , D_t$ be the strong components of $D$.  Then $D$ is $k$-AW if and only if each $D_i$ is $k$-AW.
\end{thm}

\begin{proof}
Assume that the ordering $D_1, D_2, \ldots , D_t$ is an acyclic ordering of the strong components.  We first assume each $D_i$ is $k$-AW and prove that $D$ is $k$-AW.  So let $\lambda$ be a labeling of $V(D)$ with labels in $\mathbb{Z}_k$.  We then win the game by applying the following variation on Algorithm~\ref{duh}.

\begin{itemize}
\item Toggle the vertices of $D_1$ so that all vertices have label 0.  This can be done since $D_1$ is $k$-AW.
\item Recursively, once the vertices of $D_m$ have been toggled so as to leave each vertex in $D_i$ with label 0 for each $1 \le i \le m$, we toggle the vertices of $D_{m+1}$ until all vertices in $D_{m+1}$ have label 0.  As above, this can be done since $D_{m+1}$ is $k$-AW.  Note that the vertices in $D_{m+1}$ do not dominate any vertices in any of the $D_i$ for $1 \le i \le m$, so all vertices in $D_i$ for $1 \le i \le m+1$ now have label 0.
\item Continue toggling in this way until the vertices in $D_t$ have been toggled.
\end{itemize}

As noted in the algorithm, once we toggle the vertices in $D_t$, all vertices in $D_i$ for $1 \le i \le t$ have label 0.  Since the $V(D_i)$ partition $V(D)$ by Proposition~\ref{components}(\ref{partition}), all vertices in $D$ have label 0, which makes $\lambda$ $k$-winnable.  Since $\lambda$ was arbitrary, $D$ is $k$-AW.

We now assume that $D$ is $k$-AW and prove that each $D_i$ is $k$-AW.  We actually prove something slightly stronger.  We prove that for each $i$,
\begin{itemize}
\item Every labeling $\lambda$ on $V(D_i)$ is $k$-winnable.
\item For any labeling $\pi$ on $V(D)$, if we apply the winning toggles to each $D_i$ in order (i.e. $D_1$ first, $D_2$ second, etc), then at the end of toggling the vertices of $D_i$, all vertices in $D_j$ have label 0 for all $1 \le j \le i$.
\end{itemize}

We prove this by induction.  For $i=1$, let $\lambda$ be any labeling of $D_1$.  We can extend $\lambda$ to $V(D)$ by defining $\lambda(v)$ to be any element of $\mathbb{Z}_k$ (it does not matter which element) for all $v \notin V(D_1)$.  Since $D$ is $k$-AW, this extended labeling is $k$-winnable.  Thus, we can toggle the vertices of $D$ so that every vertex has label 0.  Consider the toggles of this winning strategy that are done on $V(D_1)$.  None of the vertices outside of $V(D_1)$ dominate any of the vertices in $V(D_1)$.  Thus, in order for all vertices of $D$ (including all vertices of $D_1$) to have label 0 at the end of the winning toggling, all vertices of $D_1$ must have label 0 after we have toggled the vertices in $D_1$.  Thus, $\lambda$ is $k$-winnable, and so $D_1$ is $k$-AW.

Note that if $\pi$ is any labeling of $D$, the above argument can be used to prove that any winning toggling on $D$ that we restrict to $V(D_1)$ must result in every vertex of $D_1$ having label 0.  This completes the base case.

For induction, let $\lambda$ be a labeling of $D_i$, and extend $\lambda$ to all of $V(D)$ by defining $\lambda(v)=0$ for all $v \notin V(D_i)$.  Our induction hypothesis implies that our winning toggling must include, for all $1 \le m \le i-1$, all vertices in $D_m$ having labeling 0 after their vertices are toggled.  But since they begin with their vertices having label 0 to begin with, the toggling for each $D_m$ must be a solution to $N_m \textbf{x}_m = \textbf{0}$, where $N_m$ is the neighborhood matrix of $D_m$.  But each $D_m$ is $k$-AW, and so each $N_m$ is invertible.  It follows that each $\textbf{x}_m = \textbf{0}$, and so none of the vertices in any of the $D_i$'s are toggled.  Since no vertices in $V(D_m)$ for $m > i$ dominate any vertices in $D_i$, the toggles of vertices in $D_i$ must result in all vertices in $D_i$ having label 0.  Thus, $\lambda$ on $V(D_i)$ is $k$-winnable, and so $D_i$ is $k$-AW.

It now suffices to prove that for all labelings $\pi$ on $D$, when we apply a winning toggling to vertices of $D_m$, $1 \le m \le i$, in order, the vertices of each $D_m$ have label 0 once the vertices in $D_m$ have been toggled.  This is true for $1 \le m \le i-1$ by induction.  Once this has been done, we apply the argument in the previous paragraph to prove that after the vertices of $D_i$ have been toggled, the vertices of $D_i$ all have label 0.
\end{proof}

It follows from Theorem~\ref{strongreduce} that once we have determined which strongly connected digraphs are $k$-AW, we can apply this to the strong components of any digraph to determine whether or not it is $k$-AW.  So we concentrate our efforts on strongly connected digraphs.

\section{Using Feedback Arc Sets to Determine Winnability} \label{feedback}

The remainder of our paper will be focused on winnability in strongly connected tournaments.  An important tool in our results will be the use of feedback arc sets.  We use the following definitions.

\begin{Def}
Let $D$ be a digraph, and let $\sigma = v_1, v_2, \ldots , v_n$ be an ordering of the vertices in $D$.
\begin{enumerate}
\item The \emph{feedback arc set with respect to $\sigma$} is the set of all arcs $v_jv_i \in A(D)$ such that $i<j$.
\item A \emph{feedback arc set} is a feedback arc set with respect to some ordering of the vertices of $D$.
\item A \emph{minimum feedback arc set} is a feedback arc set of minimum cardinality among all possible orderings of the vertices of $D$.
\item If an arc $vw$ is in a feedback arc set $S$, we call $v$ a \emph{tail feedback vertex} and $w$ a \emph{head feedback vertex}.
\end{enumerate}
\end{Def}

Our definition of feedback arc sets is equivalent to the definition given in \cite{Isaak/Narayan:04}, which defines a feedback arc set as a set of arcs that when reversed makes the resulting graph acyclic.  These are called \emph{reversing sets} in \cite{Barthelemy:95}, a term apparently no longer in use.  The standard definition of feedback arc sets appears to be the one given in \cite{Barthelemy:95} and \cite{Bang-Jensen/Gutin}, which replaces \emph{reversed} in the above definition with \emph{removed}.  This is a more general definition.  Note that the definition of \emph{minimum} feedback arc set is equivalent with both definitions of feedback arc sets.


If we have a feedback arc $v_{i+1}v_i$ between consecutive vertices in an ordering of the vertices of $D$, it is clear that reversing the order of $v_i$ and $v_{i+1}$ does not create any new feedback arcs, and thus has fewer feedback arcs than the original ordering.  Thus, when we have a minimum feedback arc set, we get the following.

\begin{prop} \label{feedbackprop}
Let $D$ be a digraph, let $S$ be a minimum feedback arc set, and let $\sigma = v_1, v_2, \ldots , v_n$ be an ordering of the vertices whose feedback arc set is $S$.  If $v_jv_i \in S$, then $j-i \ge 2$.
\end{prop}

Finally, we see how feedback arc sets help us understand the result of applying Algorithm~\ref{duh} to an ordering of the vertices in the $k$-lights out game.

\begin{lem} \label{setupwin}
Let $D$ be a digraph with labeling $\lambda: V(D) \rightarrow \mathbb{Z}_k$; let $\sigma$ be an ordering of $V(D)$; and let $S$ be the feedback arc set with respect to $\sigma$.
\begin{enumerate}
\item \label{standardform} The vertices of $D$ can be toggled in such a way that all non-head feedback vertices have label 0.
\item \label{variables} Suppose $\lambda(v_i)=0$ for all non-head feedback vertices $v_i$.  If we toggle each vertex in the order given by $\sigma$ such that at the end of the toggling, each non-head feedback vertex is still 0, then each non-head feedback vertex $v_i$ is toggled $-b_i$ times, where $b_i$ is the label of $v_i$ after $v_{i-1}$ is toggled.
\end{enumerate}
\end{lem}

\begin{proof}
For (\ref{standardform}), we apply Algorithm~\ref{duh} to $D$ with the labeling $\lambda$. Let $v \in V(D)$ be a non-head feedback vertex.  After we finish toggling $v$, it has a label of 0.  Since no vertex toggled after $v$ dominates $v$, this label will not change for the remainder of the algorithm.  Thus, the final labeling for $v$ is 0, which completes the proof.  Note that if any tail feedback vertex is toggled as part of the algorithm, it may cause the final label of the vertex it dominates (i.e. a head feedback vertex) to have a nonzero final labeling.

For (\ref{variables}), $v_i$ is toggled right after $v_{i-1}$ is toggled, which means it is toggled when it has label $b_i$.   Since $v_i$ is a non-head feedback vertex, all vertices toggled after $v_i$ do not dominate $v_i$.  Thus, $v_i$ must have label 0 right after it is toggled, which can only happen if it is toggled $-b_i$ times.  The result follows directly.
\end{proof}

\section{Winnability in Strongly Connected Tournaments} \label{strong}
In this section, we focus our attention on tournaments that are strongly connected.  A \emph{tournament} is a digraph $D$ where for every distinct pair of vertices $v,w \in V(D)$, either $vw \in A(D)$ or $wv \in A(D)$, but not both.  Our results in this section will depend on the structure of the subdigraph $D_S$ that is arc-induced from a minimum feedback arc set $S$.  We consider subdigraphs $D_S$ that belong to the following classes.

\begin{Def}
Let $D$ be a digraph.
\begin{enumerate}
\item We say that $D$ is a \emph{directed path} if we can order $V(D) = \{ v_1, v_2, \ldots , v_n \}$ in such a way that $A(D) = \{ v_iv_{i+1} : 1 \le i \le n-1 \}$
\item If $s,t \ge 0$, we say that $D$ is an \emph{$(s,t)$-directed star} if we have $V(D) = \{ v, v_1, v_2, \ldots , v_s, w_1, w_2 , \ldots , w_t \}$ and $A(D) = \{ v_iv, vw_j : 1 \le i \le s$ and $1 \le j \le t \}$.
\end{enumerate}
\end{Def}
Note that if we ignore the directions of the arcs in a directed star, it becomes an ordinary star graph.

Our first main result determines winnability of strongly connected tournaments that have minimum feedback arc sets that arc-induce directed paths.  In the proof, we encounter the following sequence of matrices.  For each $n \ge 1$, we define the $n \times n$ matrix $A_n$ as follows.  Let $A_1 = [2]$ and $A_2 = \left[ \begin{matrix} 1 & 1 \\ -1 & 2 \end{matrix} \right]$.  For each $n \ge 3$, we define
\begin{equation*}
A_n = \left[ \begin{matrix}
1 & 1 &  0 & 0 & \hdots & 0 \\ 
-1 & 1 & 1 & 0 & \hdots & 0 \\ 
0 & -1 & 1 & 1 & \hdots & 0 \\ 
\vdots   & \vdots  &  \ddots & \ddots  & \ddots & \vdots \\ 
 0 & 0 & \hdots &  -1       & 1 &  1 \\
 0 & 0 & \hdots & 0 & -1 & 2
\end{matrix} \right]
\end{equation*}

\noindent Also, recall the sequence of Fibonacci numbers $F_n$ given by $F_0 = 0$, $F_1=1$ and $F_n = F_{n-1}+F_{n-2}$ for all $n \ge 2$.

\begin{lem} \label{fiblem}
For each $n \ge 1$, $\det(A_n)=F_{n+2}$.
\end{lem}

\begin{proof}
For $n=1,2$, the result is easy to check.  For $n \ge 3$, we row-reduce $A_n$ to an echelon form that has the same determinant as $A_n$ and has determinant $F_{n+2}$.

We first prove by induction that for each $1 \le i \le n-1$, we can row-reduce $A_n$ to a matrix with the same determinant as $A_n$ such that
\begin{itemize}
\item For $1 \le m \le i-1$, the leading entry of row $m$ is 1 and is located in column $m$.
\item Row $i$ will be $[\begin{matrix} 0 & \hdots & 0 & F_{i+1} & F_i & 0 & \hdots & 0 \end{matrix}]$, where $F_{i+1}$ is in column $i$.
\end{itemize}
For $i=1$, the first two entries of row 1 are both 1 (i.e. $F_2$ and $F_1$), which proves the base case.  If we assume our result for $i$, we already have the leading entry of 1 in row $m$, column $m$ for $1 \le m \le i-1$.  For the inductive step, we need only concern ourselves with rows $i$ and $i+1$.  Note that these rows are currently the following, where each leading entry is in column $i$.
\begin{equation*}
\left[ \begin{matrix} 0 & \hdots & 0 & F_{i+1} & F_i & 0 & 0 & \hdots & 0 \\
0 & \hdots & 0 & -1 & 1 & 1 & 0 & \hdots & 0 \end{matrix} \right]
\end{equation*}
We now multiply the lower row by $-1$ and switch the two rows.  By Proposition~\ref{Brownprop}, this multiplies the determinant by $-1 \cdot -1 = 1$, so the determinant is unchanged.  We get
\begin{equation*}
\left[ \begin{matrix} 0 & \hdots & 0 & 1 & -1 & -1 & 0 & \hdots & 0 \\
0 & \hdots & 0 & F_{i+1} & F_i & 0 & 0 & \hdots & 0 \end{matrix} \right]
\end{equation*}
We now add $-F_{i+1}$ times row $i$ to row $i+1$.  By Proposition~\ref{Brownprop}, this leaves the determinant unchanged.  We get
\begin{equation*}
\left[ \begin{matrix} 0 & \hdots & 0 & 1 & -1 & -1 & 0 & \hdots & 0 \\
0 & \hdots & 0 & 0 & F_i+F_{i+1} & F_{i+1} & 0 & \hdots & 0 \end{matrix} \right]
\end{equation*}
Since $F_i+F_{i+1}=F_{i+2}$, the result follows.

We now apply this result to $i=n-1$.  What we have at that point is a matrix whose first $n-2$ rows are in echelon form with 1's on the diagonal, and whose last two rows are
\begin{equation*}
\left[ \begin{matrix} 0 & \hdots & 0 & F_n & F_{n-1} \\
0 & \hdots & 0 & -1 & 2 \end{matrix} \right]
\end{equation*}
As before, we multiply the last row by $-1$ and switch rows (which again leaves the determinant unchanged).  We get
\begin{equation*}
\left[ \begin{matrix} 0 & \hdots & 0 & 1 & -2 \\
0 & \hdots & 0 & F_n & F_{n-1} \end{matrix} \right]
\end{equation*}
We now add $-F_n$ times row $n-1$ to row $n$ to get
\begin{equation*}
\left[ \begin{matrix} 0 & \hdots & 0 & 1 & -2 \\
0 & \hdots & 0 & 0 & 2F_n+F_{n-1} \end{matrix} \right]
\end{equation*}
We have $2F_n+F_{n-1} = F_n+(F_n+F_{n-1}) = F_n + F_{n+1} =F_{n+2}$.  This results in an upper diagonal matrix.  We find the determinant by multiplying the diagonal entries to get $1 \cdots 1 \cdot F_{n+2} = F_{n+2}$.  This proves the lemma.
\end{proof}

This leads us to our first main theorem.

\begin{thm} \label{paththm}
Let $D$ be a strongly connected tournament with minimum feedback arc set $S$ such that $D_S$ is a directed path.  If $|S|=m$, then $D$ is $k$-AW if and only if $\gcd(k,F_{m+2})=1$.
\end{thm}

\begin{proof}
Let $\sigma = v_1, v_2, \ldots , v_n$ be an ordering of $V(D)$ such that $S$ is the feedback arc set with respect to $\sigma$.  Since $D_S$ is a directed path with $m$ arcs, there are vertices $v_{i_1}, v_{i_2}, \ldots , v_{i_{m+1}}$ such that $S = \{ v_{i_{t+1}}v_{i_t} : 1 \le t \le m \}$.  By Lemma~\ref{feedbackprop}, $i_{t+1}-i_t \ge 2$ for each $t$, and so $v_{i_t+1}$ is not a head feedback vertex for any $1 \le t \le m$.  Also, since $D$ is strongly connected, $i_1=1$ and $i_{m+1}=n$.

We first assume that $D$ is $k$-AW.  Let $\lambda: V(D) \rightarrow \mathbb{Z}_k$ be a labeling.   By Lemma~\ref{setupwin}(\ref{standardform}), we can assume $\lambda(v)=0$ for all non-head feedback vertices $v$.  Since $\lambda$ is winnable, for each $1 \le t \le m$, we can let $x_t$ be the number of times we toggle $v_{i_t}$ as part of winning the game.  We apply these toggles in order as in Lemma~\ref{setupwin}(\ref{variables}).  The first vertex toggled is $v_{i_1}=v_1$, which we toggle $x_1$ times.  Since $v_1$ dominates every vertex except $v_{i_2}$, this increases the label of every vertex except $v_{i_2}$ by $x_1$ and leaves the label of $v_{i_2}$ unchanged.  By Lemma~\ref{setupwin}(\ref{variables}), since $v_2$ is not a head feedback vertex and now has label $x_1$, $v_2$ must be toggled $-x_1$ times to win the game.  This leaves all vertices $v$ with their original label $\lambda(v)$ except for $v=v_{i_2}$, which now has label $\lambda(v_{i_2})-x_1$.  Each $v_i$ with $2 < i < i_2$ now has label 0, so Lemma~\ref{setupwin}(\ref{variables}) implies they do not get toggled at all.  When $v_{i_2}$ gets toggled $x_2$ times, that leaves us with $v_1=v_{i_1}$ having label $\lambda(v_{i_1})+x_1+x_2$ and $v_{i_2}$ having label $\lambda(v_{i_2})-x_1+x_2$.  Note that when $v_{i_2}$ is toggled, every vertex $v_i$ with $i>i_2$ has label $\lambda(v_i)+x_2$, except $v_{i_3}$, whose label remains $\lambda(v_{i_3})$.

We apply induction to toggling the remaining vertices.  We have quite a few induction hypotheses.  First, after we toggle $v_{i_r}$ for $2 \le r \le m-1$, the labels of all vertices $v_i$ for $1 \le i \le i_r$ are as follows.
\begin{itemize}
\item If $v_i$ is a non-head feedback vertex, then the label is 0.
\item If $i=1$, then the label is $\lambda(v_{i_1})+x_1+x_2$.
\item If $2 \le t \le r-1$, then the label is $\lambda(v_{i_t})-x_{t-1}+x_t+x_{t+1}$.  Note:  This part of the induction hypothesis will be clear once we complete the inductive step.
\item $v_{i_r}$ has a label of $\lambda(v_{i_r})-x_{r-1}+x_r$.
\end{itemize}
Our induction hypothesis also assumes that the labels of all vertices $v_i$ with $i > i_r$ have label $\lambda(v_i)+x_r$, except $v_{i_{r+1}}$, whose label remains $\lambda(v_{i_{r+1}})$.

Now we are ready to continue toggling.  Our next toggle is $v_{i_r+1}$, which by induction has label $x_r$.  By Lemma~\ref{setupwin}(\ref{variables}), $v_{i_r+1}$ must be toggled $-x_r$ times.  As in the base case, all vertices $v_i$ with $i>i_r$ are returned to their original labels except for $v_{i_{r+1}}$, which now has label $\lambda(v_{i_{k+1}})-x_k$.  As in the base case, all non-head feedback vertices between $v_{i_r}$ and $v_{i_{r+1}}$ have label 0, so they are not toggled at all.  Then $v_{i_{r+1}}$ is toggled $x_{r+1}$ times.  This gives $v_{i_r}$ a label of $\lambda(v_{i_r})-x_{r-1}+x_{i_r}+x_{i_{r+1}}$ and $v_{i_{r+1}}$ a label of $\lambda(v_{i_{r+1}})-x_{i_r}+x_{r+1}$.  Furthermore, this increases the labels of all vertices $v_i$ with $i>i_{r+1}$ by $x_{r+1}$, so they have label $\lambda(v_i)+x_{r+1}$.  Since $v_{i_{r+2}}v_{i_{r+1}} \in A(D)$, $v_{i_{r+2}}$ still has label $\lambda(v_{i_{r+2}})$.  This completes the induction argument.

Once we have toggled $v_m$, we apply this process one more time to finish off the game.  Since $v_{i_m}$ was toggled $x_m$ times, that gives $v_{i_m+1}$ a label of $x_m$.  As before, $v_{i_m+1}$ is toggled $-x_m$ times.  This leaves all non-head feedback vertices with a label of 0, and leaves $v_{i_{m+1}}=v_n$ (which is not a head feedback vertex) with a label of $-x_m$.  In order to end up with a label of 0, $v_{i_{m+1}}$ is toggled $x_m$ times, leaving $v_{i_m}$ with a label of $\lambda(v_{i_m})+x_{m-1}+2x_m$.

At this point, all labels must be 0, and so we must have
\begin{itemize}
\item $\lambda(v_{i_1})+x_1+x_2=0$
\item $\lambda(v_{i_t})-x_{t-1}+x_t+x_{t+1}=0$ for $1 \le t \le m-1$
\item $\lambda(v_{i_m})+x_{m-1}+2x_m=0$
\end{itemize}
Notice that if we write this system of linear equations in matrix form, we get $A_m \textbf{x} = -\textbf{b}$, where $\textbf{x}[t]=x_t$, $\textbf{b}[t]=\lambda(v_{i_t})$, and $A_m$ is the matrix from Lemma~\ref{fiblem}.  So this matrix equation must have a solution for all $\textbf{b}$.  Conversely, any solution to the matrix equation gives us a winning toggling for the game.  Thus, $D$ is $k$-AW if and only if the matrix equation $A_m \textbf{x} = -\textbf{b}$ has a solution for every $\textbf{b}$.  This occurs precisely when $A_m$ is invertible, which by Proposition~\ref{Brownprop} occurs precisely when $\det(A_m)$ is a unit in $\mathbb{Z}_k$.  Since $\det(A_m)=F_{m+2}$ by Lemma~\ref{fiblem}, this occurs precisely when $\gcd(k,F_{m+2})=1$.
\end{proof}

We now consider feedback arc sets $S$ where $D_S$ is a directed star.  Recall that the matrices in Lemma~\ref{fiblem} were helpful in the proof of Theorem~\ref{paththm}.  The following matrices play a similar role when $D_S$ is a directed star.  

\begin{lem} \label{lowertriangular}
For each $n \ge 1$, define $L_n$ to be the lower triangular matrix where each entry at or below the main diagonal is 1.  Then $L_n$ can be row-reduced to $I_n$ by, starting from the bottom row and working upwards, subtracting each row from the row below it.
\end{lem}

\begin{proof}
Each pair of consecutive rows (say, the $i^{\hbox{th}}$ row and the $(i+1)^{\hbox{th}}$ row) looks like the following.
\begin{equation*}
\left[ \begin{matrix} 1 & \cdots & 1 & 0 & 0 & \cdots & 0 \\
1 & \cdots & 1 & 1 & 0 & \cdots & 0 \end{matrix} \right]
\end{equation*}
Subtracting the top row from the bottom row results in the bottom row having a 1 in the $(i+1)^{\hbox{th}}$ column and zeros elsewhere.  Repeating this process on $L_n$ clearly results in $I_n$.
\end{proof}

Here is some more helpful terminology for when $D_S$ is a directed star.

\begin{Def}
Let $D$ be a digraph, let $\sigma = v_1, v_2, \ldots , v_n$ be an ordering of $V(D)$, and let $S$ be the feedback arc set with respect to $\sigma$.  A \emph{head feedback interval} (resp. tail feedback interval) is a nonempty subset of $V(D)$ of the form $\{ v_i, v_{i+1}, \ldots , v_{j-1}, v_j \}$, where each $v_t$ with $i \le t \le j$ is a head feedback vertex (resp. tail feedback vertex), and both $v_{i-1}$ and $v_{j+1}$ are not head feedback vertices (resp. tail feedback vertices).  A \emph{feedback interval} is a set of vertices that is either a head feedback interval or a tail feedback interval (or both).
\end{Def}

Note that if $D_S$ is a directed star, the feedback intervals $F_1, F_2, \ldots , F_t$ can be ordered so that if $i<j$, then all vertices in $F_i$ come before all vertices in $F_j$.  Also, the vertices in the head feedback intervals come before the vertices in the tail feedback vertices, and the set containing the central vertex is both a head feedback interval and a tail feedback interval.  Thus, we can order the feedback intervals as above, but we can now write them as $H_1, H_2, \ldots , H_r, \{ v \} , T_1, T_2, \ldots , T_s$, where the $H_i$ are the head feedback intervals, the $T_j$ are the tail feedback intervals, and $v$ is the central vertex.  In this case, we have $T_j \rightarrow v \rightarrow H_i$.

We are now ready to prove our result for when $D_S$ is a directed star.

\begin{thm}
Let $D$ be a strongly connected digraph, let $\sigma = v_1, v_2, \ldots , v_n$ be an ordering of the vertices, and let $S$ be a minimum feedback arc set with respect to $\sigma$ such that $D_S$ is a directed star.  Let $m$ be the number of feedback intervals of $S$.  Then $D$ is $k$-AW if and only if $\gcd(k,m)=1$.
\end{thm}

\begin{proof}
Since $D_S$ is a directed star, we can partition $V(D_S)$ into the sets $H_1, \ldots , H_r$, $\{ v \}, T_1, \ldots , T_s$ as in the paragraph preceding the theorem statement.  Note that these are all the feedback intervals (where $\{v\}$ is both a head feedback interval and a tail feedback interval), so $m=r+s+1$.  Since $D$ is strongly connected, $v_1 \in H_1$ and $v_n \in T_s$.  We proceed similarly as in Theorem~\ref{paththm}.  We first assume that $D$ is $k$-AW.  Let $\lambda: V(D) \rightarrow \mathbb{Z}_k$ be a labeling.   By Lemma~\ref{setupwin}(\ref{standardform}), we can assume $\lambda$ is zero on all non-head feedback vertices.  Since $\lambda$ is winnable, for each head feedback vertex $u$, we can let $x_u$ be the number of times we toggle $u$ in a winning toggling.  For each $H_i$, let $y_i = \sum_{u \in H_i} x_u$. 

We apply these toggles in order as in Lemma~\ref{setupwin}(\ref{variables}).  Since $v_1 \in H_1$, we begin with the vertices in $H_1$.  Since the vertices in $H_1$ dominate every vertex after $H_1$ besides $v$, the vertex right after the last vertex in $H_1$ has label $y_1$, and so Lemma~\ref{setupwin}(\ref{variables}) implies we must toggle that vertex $-y_1$ times.  This results in the following.
\begin{itemize}
\item The vertex $v$ dominates each vertex in $H_1$, but not the vertex after the last vertex in $H_1$.  Thus, $v$ has label $-y_1$.
\item Each vertex $u$ after the vertices in $H_1$ that is not $v$ has its label increased by $y_1$ by toggling the vertices in $H_1$ and then reduced by $y_1$ by the vertex following the vertices in $H_1$.  Thus, the label of $u$ is $\lambda(u)$.
\item Each $w \in H_1$ is dominated by all vertices in $H_1$ that come before it.  Thus, each $w \in H_1$ has label $\displaystyle \sum_{u \in H_1, u \le w} x_u$, where $u \le w$ means that $u$ comes at or before $w$ in the ordering $\sigma$.
\end{itemize}
We apply induction to toggling all vertices before $v$.  Once we toggle the vertices before $H_t$ for $t \ge 2$, by induction the labels of all vertices are as follows.
\begin{itemize}
\item $v$ has label $\sum_{i=1}^{t-1}-y_i$.
\item Each vertex $u$ after the vertices in $H_{t-1}$ besides $v$ has label $\lambda(u)$.  In particular, the vertices between $H_{t-1}$ and $H_t$ have label 0.
\item For each $i <t$, each vertex $w$ in $H_i$ has label $\displaystyle \sum_{u \in H_i, u \le w} x_u$.
\end{itemize}
After we toggle the vertices of $H_t$, the vertex right after the vertices in $H_t$ has label $y_t$.  By Lemma~\ref{setupwin}(\ref{variables}), we toggle this vertex $-y_t$ times.  We can now argue as in the $t=1$ case that $v$ has label $\sum_{i=1}^{t}-y_i$; each vertex $u$ after the vertices in $H_t$ besides $v$ has label $\lambda(u)$; and each vertex $w$ in $H_t$ has label $\displaystyle \sum_{u \in H_t, u \le w} x_u$.  This completes the induction argument.  Applying this to $t=r$, we get
\begin{itemize}
\item The label of $v$ is $\sum_{i=1}^r-y_i = \sum_{u \ne v} -x_u$.
\item Every non-head feedback vertex $u$ has label $\lambda(u)=0$.
\item Every vertex $w \in H_i$ for all $1 \le i \le r$ has label $\displaystyle \sum_{u \in H_i, u \le w} x_u$.
\end{itemize}
We now finish the game by toggling $v$ $x_v$ times and following it to its conclusion.  Since $v$ dominates each head feedback vertex, each vertex $w \in H_i$ now has label $x_v + \sum_{u \in H_i, u \le w} x_u$ for all $1 \le i \le r$.  The rest of the argument follows very much as it did for the $H_i$.  Using Lemma~\ref{setupwin}(\ref{variables}), we use a similar induction argument as above to prove that each vertex that follows the last vertex in each $T_j$ (including $\{ v \}$, which is also a tail feedback interval) has a label of $x_v$ right before it is toggled, and therefore must be toggled $-x_v$ times.  This returns labels of all vertices except the tail feedback vertices (which get a label of $-x_v$) to 0.  Then the first vertex of each $T_j$ must be toggled $x_v$ times to reset all tail feedback vertices to 0.  But since each tail feedback vertex dominates $v$, it adds $x_v$ to the label of $v$.  Thus, at the end of the toggling of all vertices, the label of $v$ is increased $x_v$ for toggling $v$ and is increased by $x_v$ for toggling the vertices in each $T_i$.  Thus, $v$ has label $(s+1)x_v + \sum_{u \ne v} -x_u$.

At this point, we have won the game, so each vertex has label 0.  As in Theorem~\ref{paththm}, this leads to a system of linear equations $A \textbf{x} = -\textbf{b}$, where
\begin{equation*}
A = \left[ \begin{matrix}
L_{|H_1|} & 0 & 0 & 0 & 1 \\
0 & L_{|H_2|} & 0 & 0 & 1 \\
\vdots & \vdots & \ddots & \vdots & \vdots \\
0 & 0 & \hdots & L_{|H_r|} & 1 \\
-1 & -1 & \hdots & -1 & s+1
\end{matrix} \right]
\end{equation*}
and $\textbf{b}[i]=\lambda(v_i)$.  Note that each $L_{|H_i|}$ comes from the labels of all the $w \in H_i$ (i.e. $x_v + \sum_{u \in H_i, u \le w} x_u$), specifically the $\sum_{u \in H_i, u \le w} x_u$ part of the label.  Each of the subcolumns of 1's in the last column comes from the $x_v$ part of the label of $w$.

Conversely, any solution to $A \textbf{x} = -\textbf{b}$ will give us a winning toggling for the game, so $D$ is $k$-AW if and only if $A$ is invertible.  By Proposition~\ref{Brownprop}, this occurs precisely when $\det(A)$ is a unit in $\mathbb{Z}_k$.  We compute the determinant by row-reducing $A$ to an echelon form.  By Lemma~\ref{lowertriangular}, we can row-reduce each $L_{|H_i||}$ by subtracting each row from the one below it, starting with the bottom two rows.  If we extend this to the rows of $A$, every 1 in the last column gets canceled out except the one at the top of each block of 1's across from $L_{|H_i|}$.  This also does not affect the determinant by Proposition~\ref{Brownprop}.  Thus, we are left with the following matrix.
\begin{equation*}
\left[ \begin{matrix}
I_{|H_1|} & 0 & 0 & 0 & \textbf{e}_{|H_1|} \\
0 & I_{|H_2|} & 0 & 0 & \textbf{e}_{|H_2|} \\
\vdots & \vdots & \ddots & \vdots & \vdots \\
0 & 0 & \hdots & I_{|H_r|} & \textbf{e}_{|H_r|} \\
-1 & -1 & \hdots & -1 & s+1
\end{matrix} \right]
\end{equation*}
where each $\textbf{e}_{|H_i|}$ is a column vector whose first entry is 1, and whose other entries are 0.  We complete the row reduction by adding every row above the last row to the last row, which again does not affect the determinant.  This cancels out  all the $-1$ entries.  Also, each $\textbf{e}_{|H_i|}$ adds 1 to the last entry in the last row.  Since there are $r$ such $\textbf{e}_{|H_i|}$'s, this results in
\begin{equation*}
\left[ \begin{matrix}
I_{|H_1|} & 0 & 0 & 0 & \textbf{e}_{|H_1|} \\
0 & I_{|H_2|} & 0 & 0 & \textbf{e}_{|H_2|} \\
\vdots & \vdots & \ddots & \vdots & \vdots \\
0 & 0 & \hdots & I_{|H_r|} & \textbf{e}_{|H_r|} \\
0 & 0 & \hdots & 0 & r+s+1
\end{matrix} \right]
\end{equation*}
Since this is an upper diagonal matrix, we multiply the entries on the diagonal to get the determinant, which is $r+s+1=m$.  This is a unit precisely when $\gcd(k,m)=1$.  This completes the proof.
\end{proof}

\section{Conclusion}

We would like to classify winnability in all digraphs.  However, here are some questions that are perhaps more manageable.
\begin{itemize}
\item Theorem~\ref{strongreduce} reduces the question of being $k$-AW to strongly connected graphs.  However, little about connectivity was used in the results in Section~\ref{strong}.  Perhaps stronger results can be obtained with more sophisticated use of strong connectivity.
\item Feedback arc sets were helpful in proving our results on tournaments.  However, while our assumptions about the feedback arc sets being minimum seems a reasonable assumption, the only significant consequence of this that we used was that consecutive vertices could not be incident with the same feedback arc.  Perhaps there are more significant ways the assumption of a minimum feedback arc set can be used to obtain stronger results.
\item Finally, it appears that for a feedback arc set $S$, the isomorphism class of $D_S$ is significant (e.g. the result for directed paths was very different from the result for directed stars).  However, even within isomorphism classes, there can be other issues that affect winnability (e.g. in directed stars, whether and how the head feedback vertices and tail feedback vertices are organized into feedback intervals).  What other quirks in the ordering of vertices can affect winnability?
\end{itemize}

\bibliographystyle{amsalpha}
\bibliography{lightsout}

\end{document}